\documentclass[10pt,reqno]{amsart}
 
\usepackage{amsfonts, amssymb,stmaryrd}
\usepackage{latexsym}
\usepackage{graphicx}
\input xy
\xyoption{all}

\newtheorem{theorem}{Theorem}[section]
\newtheorem{corollary}[theorem]{Corollary}
\newtheorem{lemma}[theorem]{Lemma}
\newtheorem{proposition}[theorem]{Proposition}

\theoremstyle{definition}
\newtheorem{definition}[theorem]{Definition}
\newtheorem{example}[theorem]{Example}
\theoremstyle{remark}
\newtheorem{remark}[theorem]{Remark}

\numberwithin{equation}{section}

\newcommand{\Fin}{\mathrm{Fin}}

\newcommand{\SPAN}{\mathrm{SPAN}}
\newcommand{\supp}{\mathrm{supp}}

\newcommand{\Cost}{\mathrm{Cost}}
\newcommand{\Sub}{\mathrm{Sub}}

\newcommand{\Sgr}{\mathcal{S}\mathrm{gr}}

\newcommand{\Hom}{\mathrm{Hom}}
\newcommand{\Core}{\mathrm{Core}}

\newcommand{\HT}{\mathrm{HT}}

\newcommand{\Amm}{\mathrm{Amm}}
\newcommand{\CF}{\mathrm{CF}}

\newcommand{\PSL}{\mathrm{PSL}}

\newcommand{\Sym}{\mathrm{Sym}}

\newcommand{\Aut}{\mathrm{Aut}}
\newcommand{\Sch}{\mathrm{Sch}}
\newcommand{\IRS}{\mathrm{IRS}}

\newcommand{\Lk}{\mathrm{Lk}}

\newcommand{\HEr}{\mathrm{I}}
\newcommand{\HErE}{\mathrm{Ie}}

\newcommand{\arrow}{\rightarrow}

\newcommand{\trivgp}{\langle e \rangle}

\newcommand{\N}{\mathbf N}
\newcommand{\Z}{\mathbf Z}
\newcommand{\F}{\mathbb F}

\newcommand{\Bc}{\mathcal{B}}

\newcommand{\leftIRS}{\leftslice}

\usepackage[pdftex,bookmarks,colorlinks]{hyperref}
\begin{document}
\title{Generic IRS in free groups, after Bowen}%
\author{Amichai Eisenmann}
\author{Yair Glasner}

\maketitle




\begin{abstract}
Let $E$ be a measure preserving equivalence relation, with countable equivalence classes,  on a standard Borel probability space $(X,\Bc,\mu)$. Let $([E],d_{u})$ be the the (Polish) full group endowed with the uniform metric. If $\F_r = \langle s_1, \ldots, s_r \rangle$ is a free group on $r$-generators and $\alpha \in \Hom(\F_r,[E])$ then the stabilizer of a $\mu$-random point $\alpha(\F_r)_x \leftIRS \F_r$ is a random subgroup of $\F_r$ whose distribution is conjugation invariant. Such an object is known as an {\it{invariant random subgroup}} or an {\it{IRS}} for short. Bowen's generic model for IRS in $\F_r$ is obtained by taking $\alpha$ to be a Baire generic element in the Polish space $\Hom(\F_r, [E])$. The {\it{lean aperiodic model}} is a similar model where one forces $\alpha(\F_r)$ to have infinite orbits by imposing that $\alpha(s_1)$ be aperiodic. 

In this setting we show that for $r < \infty$ the generic IRS $\alpha(\F_r)_x \leftIRS \F_r$ is of finite index almost surely if and only if $E = E_0$ is the hyperfinite equivalence relation. For any ergodic equivalence relation we show that a generic IRS coming from the lean aperiodic model is co-amenable and core free. Finally, we consider the situation where $\alpha(\F_r)$ is highly transitive on almost every orbit and in particular the corresponding IRS is supported on maximal subgroups. Using a result of Le-Ma\^{i}tre we show that such examples exist for any aperiodic ergodic $E$ of finite cost. For the hyperfinite equivalence relation $E_0$ we show that high transitivity is generic in the lean aperiodic model.   
\end{abstract}

\section{Introduction}
Let $\Gamma$ be a countable group, $\{0,1\}^{\Gamma}$ the space of all subsets of $\Gamma$ with the compact Tychonoff topology. The collection of all subgroups $\Sub(\Gamma) \subset \{0,1\}^{\Gamma}$ is closed and therefore a compact metrizable space in its own right. The induced topology on $\Sub(\Gamma)$ is known as the {\it{Chabauty topology}} (see \cite{Chabauty:top}).  The group $\Gamma$ acts continuously by conjugation on $\Sub(\Gamma)$.
\begin{definition}
An {\it{Invariant random subgroup}}, or $\IRS$ for short, of $\Gamma$ is a $\Gamma$ invariant Borel probability measure on $\Sub(\Gamma)$. We will also say that $\Delta \in \Sub(\Gamma)$ is an {\it{invariant random subgroup}} and write $\Delta \leftIRS \Gamma$ to signify that such an invariant probability measure has been fixed and that $\Delta$ is a random subgroup chosen according to this distribution. We write $\IRS(\Gamma)$ and $\IRS^e(\Gamma)$ to denote the collection of all (resp. all ergodic) invariant measures on $\Sub(\Gamma)$.
\end{definition}
Note that $\IRS(\Gamma)$ can be viewed as a simplex in the dual Banach space $C(\Sub(\Gamma))^{*}$ and $\IRS^e(\Gamma)$ is the collection of extreme points of this simplex.

A similar notion can be defined in the setting of a locally compact group $G$ - see \cite{7_sam,7_sam:ann}. In this setting $\Sub(G)$ denotes the collection of all closed subgroups of $G$ which is again a compact metrizable space. During the last few years IRS turned out to be surprisingly useful in a wide array of mathematical branches. We refer to the following papers and the references in them \cite{AGV:Kesten_IRS, AGV:Kesten_Measureable} for spectral graph theory, \cite{7_sam} for representation theory and asymptotic invariants of Lie groups and their lattices, \cite{TD:shift_minimal,PT:stab_erg} for applications to operator algebras, and \cite{BDL:amenable_irs,G:linear_irs} for some structure theory of IRS.

\begin{example} \label{eg:pmp}
Let $\Gamma \curvearrowright (X,\Bc,\mu)$ be a measure preserving action of $\Gamma$ on a probability space. Then there is a natural map $\Phi: X \arrow \Sub(\Gamma)$ defined by $x \mapsto \Phi(x)= \Gamma_x$. Since the action is measure preserving and the map $\Phi$ is $\Gamma$ equivariant, the image of the measure $\Phi_*(\mu) \in \IRS(\Gamma)$. In short one can just say that $\Gamma_x \leftIRS \Gamma$ for every probability preserving action as above. In fact in \cite[Proposition 13]{AGV:Kesten_IRS} it is shown that this example is universal in the sense that every $\nu \in \IRS(\Gamma)$ can be obtained by such a construction whenever $\Gamma$ is finitely generated. This was generalized to locally compact groups in \cite{7_sam}.
\end{example}

The oldest, and one of the deepest results to date concerning IRS, due to Stuck-Zimmer \cite{SZ} is the full classification of $\IRS(G)$ where $G$ is a higher rank semi-simple Lie group, or a lattice thereof. In both cases every $\mu \in \IRS^e(G)$ is supported on an orbit (i.e. a conjugacy class) either of a finite central subgroup or of a discrete group $H$ of finite co-volume - when $G$ is a Lie group $H < G$ is a lattice and when $G$ is a lattice $H$ is just a finite index subgroup. Other structure results singling out groups in which the space $\IRS(\Gamma)$ is small are due to Vershik \cite{Vershik:characters} ($G = \Sym^{f}_{\infty}$ the infinite finitary permutation group), Peterson-Thom  \cite{PT:stab_erg} ($\PSL_n(k)$ where $k$ is a countable field).

Let $\F_r = \langle S \rangle$ be the non-abelian free group on $r$ generators $S = \{s_1,s_2,\ldots,s_r\}$. In contrary to the above results, Lewis Bowen \cite{bowen:irs_free} studies $\IRS(\F_r)$ using a variety of methods and finds a very rich structure.

\subsection{Bowen's model of generic IRS in $\F_r$}
One of the methods that Bowen introduces is that of a {\it{generic IRS}}. We give only a short survey of this method and refer the readers to the original paper \cite{bowen:irs_free} and also \cite{bowen:furst_ent} for more details. We use the word {\it{generic}} in the Baire category sense of the word. A generic model for an IRS will consist of a Polish (i.e. metrizable, separable and complete) topological space $\HEr$ together with a map $f: \HEr \arrow \IRS(\F_r)$. Given such a model we will look for properties of the IRS $f(x)$ that hold for every $x$ in a residual (e.g. a dense $G_{\delta})$ subset of $\HEr$.

The group $A:=\Aut(X,\Bc,\mu)$ itself has a natural Polish structure coming from the weak topology. Thus the space $\Hom(\F_r, A) \cong A^r$ is also Polish. As explained in Example \ref{eg:pmp} above to any $\alpha \in \Hom(\F_r,A)$ one can associate the $\IRS$ $\alpha(\Gamma)_x \leftIRS \Gamma$.  But as it turns out, for a residual set of $\alpha$'s, this $\IRS$ is almost surely trivial. Bowen's idea was to fix in advance a Borel equivalence relation $E \subset X \times X$ with countable classes. Assume that $\mu$ is $E$ invariant and that $E$ is aperiodic, in the sense of the following:
\begin{definition}
An equivalence relation $[F]$ is called {\it{periodic}} if its equivalence classes are finite almost surely. It is called {\it{aperiodic}} if its equivalence classes are infinite almost surely.
\end{definition}
\noindent To such an equivalence relation we can associate its full group:
\begin{definition}
Given an equivalence relation $E$ as above let $[E]$ denote the {\it{full group}} of this equivalence relation:
$$[E] = \{g : X \arrow X \ | \ g {\text{ is a Borel isomorphism and }} xEgx \ \forall x \in X \},$$
endowed with the uniform metric
$$d(\phi, \psi) = \mu(\{x \in X \ | \ \phi(x) \ne \psi(x) \} ).$$
\end{definition}
The uniform metric can be defined on the whole group $\Aut(X,\Bc,\mu)$ but it gives rise to a non-separable topological group. When restricted to the full group $[E]$ the uniform metric gives rise to a Polish group structure as proven in \cite[Proposition 3.2]{Kechris:Global_aspects}.

Once this Polish group topology has been fixed, we can define the following Polish space:
\begin{eqnarray} \label{eqn:pol_space}
&& \nonumber \Hom(\F_r, [E]) \cong [E]^r, \\
&& d_u(\alpha,\beta) = \sup_{i =1 \ldots r} \{d_u(\alpha(s_i), \beta(s_i)) \} \ \alpha,\beta \in \Hom(\F_r,[E]).
\end{eqnarray}
The Gaboriau-Levitt theory of {\it{cost}} will play an important part. $\Cost(E)$ is an invariant associated with the equivalence relation. A famous Theorem of Gaboriau \cite[Corollaire 1]{Gab:cout} shows that the cost of the orbit-equivalence relation coming from a free action of $\F_r$ is exactly $r$. Thus if $r > \Cost(E)$ the action given by any $\phi: \F_r \arrow [E]$ cannot be essentially free and consequently the associated IRS cannot be almost surely trivial.

\subsection{Main results}
In many of our theorems we address the question of when the IRS arising from a generic $\alpha \in \Hom(\F_r,[E])$ is large. The word large is considered  here in various different meanings.  The most obvious one is when the IRS is of finite index, or in other words if the $\alpha(\F_r)$-orbit equivalence relation is periodic. As it turns out, this depends on whether the equivalence relation is hyperfinite or not.
\begin{definition}
A Borel equivalence relation $E$ is called {\it{hyperfinite}} if it can be expressed as the ascending union of Borel sub-equivalence relations $E = \cup_n F_n$ where $F_n \subset F_{n+1}$ and each $F_n$ is periodic.
\end{definition}
It is a famous result of Orenstein-Weiss \cite{OW} that every action of an amenable group gives rise to a hyperfinite equivalence relation.
\begin{theorem} \label{thm:periodic_hf}
Let $E$ be a Borel, measure preserving equivalence relation with countable equivalence classes on a standard probability space $(X,\Bc,\mu)$. For $\alpha \in \Hom(\F_r,[E])$ consider the associated IRS $\alpha(\F_r)_x \leftIRS \F_r$ and its index $I_{\alpha}:=[\F_r: \alpha(\F_r)_x]$ - an integer valued random variable. Then for a generic $\alpha \in \Hom(\F_r, [E])$ the following holds:
\begin{enumerate}
\item \label{itm:inf_gen} {\bf{For $\bf{r = \infty}$.}} $\alpha(F_{\infty})$ spans the equivalence relation. In particular  $I_{\alpha} = \infty$ almost surely.
\item \label{itm:hyperfinite} {\bf{For $\bf{2 \le r < \infty}$}}. The following conditions are equivalent:
\begin{itemize}
\item $E$ is hyperfinite.
\item $I_{\alpha} < \infty$ almost surely.
\end{itemize}
\end{enumerate}
\end{theorem}

One of the main reasons for considering generic models for IRS is as a rich source of examples. As such finite index IRS are not interesting. It is therefore desirable to find a generic model that forces the IRS to be of infinite index. The most natural way to do this is to set $\Hom'(\F_r,[E]) := \{\alpha \in \Hom(\F_r,[E]) \ | \ [\F_r: \alpha(\F_r)_x] = \infty \ {\text{almost surely}}\}$. This turns out to be a $G_{\delta}$ subset of $\Hom(\F_r,[E])$ and therefore, even though it might be meager, it is a Polish space in its own right and can serve as a model for infinite index $\IRS$. This was the strategy treated by Bowen, we adopt a different model which is perhaps less natural but much easier to work with. We force the action to be non-periodic by demanding that the first generator $S:=\alpha(s_1)$ already is aperiodic.
\begin{definition} \label{def:easy_ap}
For every $r$ and every aperiodic, measure preserving equivalence relation $E$, with countable classes, on a standard probability space we define the {\it{lean aperiodic model}} to be
$$\HEr = \HEr(r,E) = \{ \alpha \in \Hom(\F_r, [E]) \ | \ \alpha(s_1) {\text{ is aperiodic}} \} \subset \Hom(\F_r,[E]).$$
\end{definition}
\begin{remark} \label{rem:ap->erg}
Theorem 3.6 in \cite{Kechris:Global_aspects} says that whenever $E$ is ergodic the set $\HErE := \{\alpha \in \HEr \ | \ \alpha(s_1) {\text{ is ergodic}} \}$ is residual in $\HEr$. Thus in this case any generic statement in $\HErE$ is true also in $\HEr$ and vice versa. 
\end{remark}
Co-amenable subgroups, as defined below, generalize finite index subgroups, in much the same way that amenable subgroups generalize finite subgroups. In particular a normal subgroup $N \lhd \Gamma$ is co-amenable iff $\Gamma/N$ is amenable. 
\begin{definition}
A subgroup $\Delta < \Gamma$ is called {\it{co-amenable}} if one and hence all of the following equivalent conditions hold: (i) there is a $\Gamma$ invariant mean on $\Gamma/\Delta$; (ii) there exists a sequence of F{\o}lner sets $F_n \subset \Gamma/\Delta$ such that $\lim_{n \arrow \infty} \left|\gamma F_n \Delta  F_n \right| /\left| F_n \right| =0$ for every $\gamma \in \Gamma$; (iii) for every continuous affine action of $\Gamma$ on a compact convex subspace $C$ of a locally convex space, the existence of a $\Delta$ fixed point implies the existence of a $\Gamma$ fixed point.
\end{definition}
Our next theorem says that for ergodic equivalence relations co-amenability is generic in the lean aperiodic model. 
\begin{theorem} \label{thm:co-am}
Let $E$ be an ergodic, measure preserving equivalence relation with countable equivalence classes on a standard Borel probability space $(X,\Bc,\mu)$. Let $\HEr$ be the lean aperiodic model defined above. Then for a residual set of $\alpha \in \HEr$ the associated IRS $\alpha(\F_r)_x \leftIRS \F_r$ is co-amenable almost surely. 
\end{theorem}
\begin{remark} \label{rem:Kaimanovich}
If $E = E_0$ is the hyperfinite equivalence relation then it was proved by Vadim Kaimanovich in \cite{Kaimanovich:amen_hf} that {\it{every}} $\alpha \in \Hom(\F_r,[E_0])$ gives rise to a co-amenable IRS $\alpha(\F_r)_x \leftIRS \F_r$.
\end{remark}
An additional indication for the largeness of a subgroup $\Delta < \Gamma$ is the extent of transitivity of the corresponding coset action $\Gamma \curvearrowright \Gamma/\Delta$. We will go directly to the highest possible transitivity level. 
 \begin{definition}
A group action on a set $\Gamma \curvearrowright \Omega$ is called {\it{highly transitive}} if it is transitive on ordered $k$-tuples of distinct points for every $k \in \N$. We say that a subgroup $\Delta < \Gamma$ is {\it{co highly-transitive}} or {\it{co-HT}} for short, if the permutation action $\Gamma \curvearrowright \Gamma/\Delta$ is highly transitive. 
\end{definition}
\begin{remark} \label{rem:ht_dense}
Equivalently an action $\Gamma \curvearrowright \Omega$ is highly transitive if the corresponding permutation representation $\Gamma \arrow \Sym(\Omega)$ has a dense image. Here $\Sym(\Omega)$ is the full symmetric group $\Sym(\Omega)$ and it is taken with the (Polish) topology of pointwise convergence. 
\end{remark}
\begin{definition}
The {\it{core}} of a subgroup $\Delta < \Gamma$ is $\Core_{\Gamma}(\Delta) = \cap_{\gamma \in \Gamma} \gamma \Delta \gamma^{-1}$ - the largest normal subgroup that it contains. If $\Core_{\Gamma}(\Delta) = \trivgp$ we say that $\Delta$ is {\it{core free}}. $\Delta$ is core free if and only if the coset action $\Gamma \curvearrowright \Gamma/\Delta$ is faithful. 
\end{definition}
A group $\Gamma$ is called {\it{highly transitive}} if it admits a faithful, highly transitive action on a set; or in other words if it contains a core free co-HT subgroup. Some countable groups are explicitly given as highly transitive permutation groups. The outstanding example is of course $\Sym^{f}(\Omega)$, the group of finitely supported permutations, and all subgroups of the full symmetric group that contain it. The first examples of highly transitive groups, that were not implicitly given as such, were  constructed by McDonough \cite{MD_free_HT} for non-abelian free groups. Many examples followed \cite{Kit:ht,MS:ht,GG:ht,HO:ht} to note a few. The last two papers seem to indicate that the family of highly transitive groups is much richer than was originally expected. In the last paper for example Hull and Osin show that every acylindrically hyperbolic group with no non-trivial finite normal subgroups is highly transitive!

\begin{definition}
A subgroup $\Delta < \Gamma$ is called {\it{profinitely dense}} (resp. {\it{pro-dense}}) if $\Delta N = \Gamma$ for every finite index (resp. for every non-trivial) normal subgroup $N \lhd \Gamma$.
\end{definition}
A group which is co-HT automatically admits a few weaker properties which are of great interest. It is maximal of infinite index and profinitely dense. If it is core free then it is also pro-dense. In fact, it is well known that the following implications hold for a subgroup $\Delta < \Gamma$.
\begin{eqnarray*}
 {\text{co-HT}} & \Rightarrow {\text{Maximal of infinite index}} & \Rightarrow {\text{profinitely dense}}\\
{\text{co-HT and core free}} & \Rightarrow {\text{Maximal and core free}} & \Rightarrow {\text{pro dense}}
\end{eqnarray*}

\noindent For the core of a generic IRS we have the following:
\begin{theorem} \label{thm:core}
Let $E$ be an ergodic measure preserving equivalence relation with countable equivalence classes on a standard Borel probability space $(X,\Bc,\mu)$. Let $\alpha$ be a generic element  in the lean aperiodic model $\HEr$. Then the associated IRS $\alpha(\F_r)_x \leftIRS \F_r$ is core free almost surely.
\end{theorem}

We do not know if a generic IRS in the lean aperiodic model is co-HT for general equivalence relations. However for the hyperfinite equivalence relation we can prove that this is the case. 
\begin{theorem} \label{thm:coHT}
Let $E_0$ be the ergodic aperiodic hyperfinite equivalence relation and $\HEr$ the corresponding lean aperiodic model. For a generic $\alpha \in \HEr$ the associated IRS $\alpha(\F_r)_x \leftIRS \F_r$ is co-HT almost surely.
\end{theorem}
\begin{remark} \label{rem:LM}
A much more general theorem was proved by Le-Ma\^itre \cite[Theorem 5.1]{LM:lean} who obtains the same results the much wider class of equivalence relations with $\Cost(E)=1$. Even though Le-Ma\^itre's theorem is more general we decided to leave our own statement because the proof in this specific case is very easy and might be of interest. 
\end{remark}

For general aperiodic Ergodic, Borel equivalence relation $E$ with countable classes there exist $\alpha\in \Hom(\F_r,[E])$ such that the associated IRS is core free and co-HT almost everywhere. To see this note that every such $E$ contains an ergodic aperiodic hyperfinite subequivalence relation $E_0$ ( \cite[Lemma 23.2]{Kechris_Miller}). From Theorem \ref{thm:coHT} above, there is an abundance of elements in $\Hom(\F_r,[E_0])$ such that the associated IRS is highly transitive almost everywhere, but of course each such element is also an element of $\Hom(\F_r,[E])$.

It does seem more interesting to construct actions that span the whole equivalence relation, whose associated IRS are co-HT almost surely This can be done, albeit not in a generic construction. Our observation here is the following:
\begin{proposition} \label{prop:dense->ht}
Let $\Delta<[E]$ be a dense subgroup of the full group of a Borel equivalence relation $E$ on $(X,\mu)$. Then $\Delta$ acts highly transitively on almost every equivalence class of $E$.
\end{proposition}
We mentioned in Remark \ref{rem:ht_dense}, that a permutation group $\Delta < \Sym(\Omega)$ is highly transitive if and only if is dense in $\Sym(\Omega)$. The analogy with Proposition \ref{prop:dense->ht} suggests the question weather the converse to Proposition \ref{prop:dense->ht} is also true? 

One obvious corollary of Proposition \ref{prop:dense->ht} is worth pointing out. If $\Delta$  is not highly transitive then it cannot admit a dense embedding into $[E]$. For example, by \cite{HO:ht} any group that satisfies a non-trivial mixed identity, but is not an intermediate subgroup between the full symmetric group, and the finitely supported permutation group of a countable set $\Sym^{f}(\Omega) < \Delta < \Sym(\Omega)$ is not highly transitive. Consequently such a group cannot be densely embedded into a full group of an aperiodic equivalence relation $[E]$. Similarly a group that does not admit a highly transitive permutation representation on a countable set will not admit a homomorphism into $[E]$ with a dense image. 

Any $\alpha \in \Hom(\F_r,E)$ with $\alpha(\F_r)$ dense in the full group $[E]$ would give rise to a co-HT IRS $\alpha(\F_r)_x$. The existence of dense subgroups, and the exact number of generators needed in order to generate them is a delicate question. Luckily this was fully solved in two very elegant papers by Fran\c{c}ois Le-Ma\^itre \cite{Le_Maitre,LM:lean}. In the first paper he proves that if $r > \Cost(E)$, for an ergodic equivalence relation $E$, then there exists a representation $\alpha \in \Hom(\F_r, [E])$ with a dense image. In the second paper he establishes the generalization of Theorem \ref{thm:coHT} to cost one equivalence relations. As a corollary we deduce:
\begin{theorem} \label{thm:LM}
Let $E$ be an ergodic measure preserving equivalence relation with countable equivalence classes on a standard Borel probability space $(X,\Bc,\mu)$. Assume that $\Cost(E) < r$ then:
\begin{itemize}
\item There exists an action $\alpha \in \Hom(\F_r,[E])$ that generates $[E]$, and whose associated IRS $\alpha(\F_r)_x \leftIRS \F_r$ is almost surely co highly transitive.
\item If moreover $\Cost(E)=1$ then the same holds for a generic $\alpha \in \HEr$. 
\end{itemize}
\end{theorem}

As we mentioned earlier Bowen uses a slightly different generic model for aperiodic actions $\Hom'(F_r,[E]) = \{\phi \in \Hom(F_r,[E]) \ | \ \phi(F_r) {\text{ is aperiodic}}\}$. This is a $G_{\delta}$ subset of $\Hom(F_r,[E])$. As a main motivation for the introduction of this generic model he proves that when $r > \Cost(E)$ then for a generic $\phi \in \Hom'(F_r,[E])$ the associated factor map 
\begin{equation} \label{eqn:factor}
(X,\Bc,\mu) \stackrel{\Phi}{\arrow} (\Sub(F_r),\Phi_*\mu), \qquad x \mapsto \phi(F_r)_x\end{equation}
becomes an isomorphism. This is Theorem \cite[Theorem 5.3]{bowen:irs_free}. It is not difficult to go over the proof of that theorem and see that a similar proof works also for a generic $\phi \in \HEr$ in the lean aperiodic model. It was suggested by the referee that this result of Bowen combines with the results of Le-Ma\^itre and ourselves to give the following:

\begin{corollary} \label{cor:conc}
Let $E$ be a Borel equivalence relation with countable equivalence classes. Assume that $\Cost(E)=1$ on . Then $E$ is orbit equivalent to an orbit equivalence relation coming from an action of $\F_2$ on $(\Sub(\F_2),\nu)$ where $\nu$ is an appropriately chosen IRS. Moreover one can assume in addition that $\nu$ is supported on core free, co-amenable and co-HT  subgroup. 
\end{corollary}
\begin{proof}
By Bowen's theorem a generic $\phi \in \HEr$ induces an isomorphism in the factor map \ref{eqn:factor}. The fact that the resulting IRS $\nu := \Phi_*(\mu)$ is generically supported on core free, co-amenable and co-HT subgroups follows from Theorems \ref{thm:core}, \ref{thm:co-am} and \cite[Theorem 5.1]{LM:lean} respectively. By Baire's category theorem a generic $\phi$ will satisfy all of these conditions simultaneously.
\end{proof}
In particular even if we take a non hyperfinite equivalence relation of cost one we see that a generic action will have F\o lner sets in almost every orbit. Such actions on non hyperfinite equivalence relations were first constructed by Kaimanovich in\cite{Kaimanovich:amen_hf}.

 \section{Basic lemmata and definitions}
 \begin{lemma} \label{lem:rep}
Let $E$ be an ergodic equivalence relation on a standard Borel probability space $(X,\Bc,\mu)$ and let $\sigma \in [E]$. Then given any $A \in \mathcal{B}$ and $\tau \in [E]$ there exists  $\sigma_1 \in [E]$ such that $\sigma_1|_A = \tau |_A$ and $d_{u}(\sigma, \sigma_1) < 2 \mu(A)$.
 \end{lemma}
 \begin{proof} Let
 \begin{equation}
 \sigma_1(x) = \left \{
    \begin{array}{lll} \tau(x) & \qquad & x \in A \\
                                \sigma(x) & \qquad & x \not \in A \cup \sigma^{-1} \tau (A) \\
                                \eta(x) & \qquad & {\text{otherwise}}
    \end{array} \right.
 \end{equation}
 where $\eta \in [E]$ is some element taking $\sigma^{-1} \tau A \setminus A$ to $\sigma(A) \setminus \tau(A)$ the existence of which is guaranteed by the fact that $E$ is ergodic and that these two sets have the same measure.
\begin{figure}[ht] \label{fig:lem_1}
\centering \def\svgwidth{150pt}
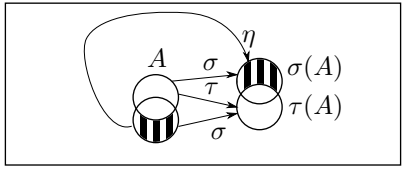 
\caption{Changing $\sigma$ on a small set $A$.}
\end{figure}
\end{proof}
\begin{lemma} \label{lem:cont}
The substitution map $\HEr \arrow [E]$ given by $\alpha \mapsto \alpha(\gamma)$ for some fixed $\gamma \in \F_r$ is continuous. In fact $d_u(\alpha(\gamma), \beta(\gamma)) \le \ell(\gamma) d_u(\alpha, \beta), \ \forall \alpha,\beta \in \HEr$, where $\ell(\gamma)$ is the word length of $\gamma$ with respect to the standard set of generators.
\end{lemma}
\begin{proof}
We prove this by induction on $\ell = \ell(\gamma)$. Let $V(\gamma) = \{x \in X \ | \ \alpha (\gamma,x) \ne \beta (\gamma,x) \}$. By definition of the uniform metric $\mu(V(s)) \le d_u(\alpha,\beta), \forall s \in S \sqcup S^{-1}$ which means that the lemma is true for $\ell = 1$. Assume now that $\gamma = \gamma' s$ where $s \in S \sqcup S^{-1}$ and $\gamma' \in \F_r$ with $\ell(\gamma') = \ell -1$, then
$$V(\gamma) \subset V(s) \cup (\alpha (s))^{-1} (V(\gamma'))$$
and the lemma follows.
\end{proof}

We recall that a Schreier graph of $\F_r$ is a $2r$-regular rooted graph $(Z,z_0) = (V,E,z_0)$ with a labeling of the oriented edges $\iota: E \arrow S \sqcup S^{-1}$ satisfying the two properties $\iota(\overline{e}) = \iota(e)^{-1}$ and $\iota^{-1} (\Lk(v)) = S \sqcup S^{-1} \ \forall v \in V$. To every $\alpha \in \Hom(\F_r, [E])$ and $x \in X$ we can associate the corresponding Schreier graph $\Sch(\alpha, x) = (V,E,x)$ with vertex set $V = \alpha(\F_r) \cdot x$ and an edge labeled $s$ connecting $v$ and $\alpha(s)(v)$ for every $v \in V$ and $s \in S \sqcup S^{-1}$. By an isomorphism of Schreier graphs we mean a graph isomorphism that respects the base vertex and the edge coloring. Let $\Sgr(r)$ be the collection of isomorphism classes of $\F_r$-Schreier graphs with the natural topology that turns it into a compact metrizable space. A natural countable basis for the topology on $\Sgr$, is given by sets of the form 
$$U(f,R) = \{Z \in \Sgr \ | \ B_Z(R) \cong B_f(R)\} \qquad f \in \Sgr, R \in \N.$$ The set of all Schreier graphs such that the ball of radius $R$ around the root is isomorphic, as a labeled graph, to the ball of the same radius in a fixed Schreier graph $f \in \Sgr$. 
\begin{lemma} \label{lem:Gd}
For a subset $\Xi \subset \Sgr(r)$ we define
$$\tilde{\Xi} = \{\alpha \in \Hom(\F_r, [E]) \ | \ \Sch(\alpha,x) \in \Xi {\text{ for $\mu$-almost every }} x \in X \}.$$ Then $\tilde{\Xi}$ is $G_{\delta}$ in $\Hom(\F_r, [E])$ whenever $\Xi$ is $G_{\delta}$ in $\Sgr(r)$.
\end{lemma}
\begin{proof}
It is enough to prove the statement if $\Xi$ is open. Indeed if $\Xi = \cap_i \Xi_i$ is a countable intersection of open sets $\Xi_i$ then
$\tilde{\Xi} = \cap_i \tilde{\Xi_i}$.
So we will assume that $\Xi$ is open. We can realize $\Xi$ as a union of basic open sets $\Xi = \cup_{j \in J} U(f_j,R_j)$, where $\{f_j \ | \ j \in J \} \subset \Sgr$ is a set of Schreier graphs and $\{R_j \ | \ j \in J \}$ corresponding radii. Set
\begin{equation*}
\tilde{\Xi}(\epsilon) = \left \{ \alpha \in \Hom(\F_r, [E]) \ | \ \mu \left \{ x \in X \ | \ \Sch(\alpha,x) \in \Xi \right \} > 1-\epsilon \right \} \\
\end{equation*}
Since $\tilde{\Xi} = \cap_{m \in \N} \tilde{\Xi}(1/m)$, it is enough to prove that $\tilde{\Xi}(\epsilon)$ is open for all $\epsilon > 0$. Now for every $\alpha \in \tilde{\Xi}(\epsilon)$ one can find a finite subset $JF \subset J$ such that $\mu\{x \in X \ | \ \Sch(\alpha,x) \in U(f_j,R_j) {\text{ for some }} j \in JF \} > 1-\eta> 1-\epsilon$, for some $\eta<\epsilon$. Let $R = \max \{R_j \ | \ j \in JF \}$ be the maximal of these radii.

For $\beta$ close enough to $\alpha$, say $d_u(\alpha,\beta) < \delta$, consider the set $$\{x \in X \ | \ \Sch(\beta,x) \in U(f_j,R_j) {\text{ for some }} j \in JF \}.$$ We want to show that if $\delta$ is sufficiently small then this set still has a measure larger than $1-\epsilon$. Let $\Omega = B_{\F_r}(2R+1) = \{\gamma \in \F_r \ | \ \ell(\gamma) < 2R+1 \}$. Consider the set
$$B = \{ x \in X \ | \ \alpha(\gamma,x) = \beta(\gamma,x) \ \forall \gamma \in \Omega \}$$ Applying Lemma \ref{lem:cont} to each $\gamma \in \Omega$ separately and taking the union bound we see that
$$\mu(B) \ge 1 - \delta (2R+1) |\Omega| \ge 1-\delta (2R+1) (2r)^{2R+1} > 1 - \eta'.$$ for some $\eta'<\epsilon-\eta$, where the last inequality is easily attained by choosing $\delta$ sufficiently small. For every $x \in B$ the two balls in the Schreier graphs are isomorphic
$B_{\Sch(\alpha,x)} (R) \cong B_{\Sch(\beta,x)}(R)$. Indeed two elements of $\gamma, \gamma' \in \Gamma$:
\begin{itemize}
\item Represent the same vertex in the Schreier graph $\Sch(\alpha,x)$ if and only if $\alpha(\gamma^{-1} \gamma', x) = x$.
\item Represent two vertices connected by an edge labeled $s$ if and only if $\alpha(\gamma^{-1} s \gamma',x) =x$.
\end{itemize}
But if $\ell(\gamma), \ell(\gamma') < R$ then all of the above discussion implies that these two quarries would yield the exact same results if we apply $\beta$ instead of $\alpha$. This finishes the proof, as clearly $\mu\{x \in X \ | \ \Sch(\beta,x) \in U(f_j,R_j) {\text{ for some }} j \in JF \})>1-\eta-\eta'>1-\epsilon$.  
\end{proof}

\section{Periodicity vs. hyperfinitness}
\noindent
In this section we prove Theorem \ref{thm:periodic_hf}. We start with a lemma of independent interest.
\begin{lemma} \label{lem:per_dense}
Let $\alpha_n,\alpha \in \Hom(\F_r,\Aut(X,\Bc,\mu))$ with $\lim_n \alpha_n \stackrel{u}{\arrow} \alpha$ in the uniform topology. If the orbits of $\alpha_n(\F_r)$ are almost surely finite for every $n$ then the equivalence relation spanned by $\alpha(\F_r)$ is hyperfinite.
\end{lemma}
\begin{proof} Let $E$ denote the equivalence relation induced by $\alpha(\F_r)$. We would like to show that there is a co-null subset $Y$  of $X$ such that the restriction of $E$ to $Y$ is hyperfinite. We may suppose, passing to  a subsequence of $\{\alpha_n\}$ if necessary, that $d_n:=d_u(\alpha_n,\alpha)<\frac{1}{2^nr}$ for all $n\in\mathbb{N}$, where $d_u$ denotes the uniform distance as usual. Let $\{s_1, \ldots, s_r\}$ be the fixed set of generators of $\F_r$. Let $E_n'$ denote the equivalence relation generated by $\alpha_n$. For every $n$ consider the  subset  $A_n':=\{x\in X \ | \ \alpha_n(s_i)(x)=\alpha(s_i)(x),\;0\leq i\leq r \}$. Clearly  $\mu(A_n')>1-\frac{1}{2^n}$. Define $A_n:=A_n'\setminus\cup_{k=n+1}^{\infty}(X\setminus A_k')$. Then $\{A_n\}$ is an ascending sequence of Borel subsets of $X$
, $\mu(A_n)>1-\frac{1}{2^{n-1}}$, and $\alpha_n(s_i)(x)=\alpha(s_i)(x)$ for all $x\in A_n$ and $1\leq i\leq r$. Let $E_n$ be the Borel equivalence relation defined as $E_{n|A_n}'$ on $A_n$ and as the identity relation on $X\setminus A_n$. Then $\{E_n\}_{n=1}^{\infty}$ is an ascending sequence of periodic equivalence relations. Fix $\gamma\in \F_r$, and suppose $\gamma=s_{i_1}s_{i_2}\ldots s_{i_k}$. Note that whenever $(x,\alpha(\gamma)(x)) \not \in E_n$, then either $(x,\alpha(s_{i_k})(x)) \not \in E_n$ or $(\alpha(s_{i_k}(x), \alpha(s_{i_{k-1}} s_{i_k})(x))) \not \in E_n$, or $\ldots$ $(\alpha(s_{i_2}\ldots s_{i_{k-1}} s_{i_k})(x),\alpha(\gamma)(x)) \not \in E_n$. Hence, for a given $n\in\mathbb{N}$, we have 
\begin{small} 
\begin{eqnarray*}
& \mu & (\{x\in X \ | \ (x,\alpha(\gamma x))\notin E_n\}) \\
&&\leq  \sum_{j=1}^k\mu(\{x\in X|(\alpha(s_{i_{j+1}}\ldots s_{i_k})x,\alpha(s_{i_j}\ldots s_{i_k})x)\notin E_n\}) \\ 
&& \leq \sum_{j=1}^k\mu(\{x\in X|(x,\alpha(s_{i_j})x)\notin E_n\}) \\
&& \leq  \sum_{j=1}^k(\mu(X\setminus A_n\cup X\setminus s_{i_j}^{-1}A_n))\leq 2k\cdot\frac{1}{2^{n-1}}
\end{eqnarray*} \end{small}
 Let $B_{\gamma}:=\{x\in X|(x,\alpha(\gamma) x)\notin E_n, \forall n\in\mathbb{N}\}$. Then by the calculation above $\mu(B_{\gamma})=0$ and $\mu(\cup_{\gamma\in \F_r}B_{\gamma})=0$. Let $Y:=X\setminus\cup_{\gamma\in \F_r}B_{\gamma}$. Then $\mu(Y)=1$ and $E_{|Y}\subset(\cup E_n)_{|Y}$, so the restriction of $E$ to the co-null subset $Y$ is hyperfinite, proving the lemma.
 \end{proof}
\begin{remark}
Note that in the lemma above we may take any finitely generated group instead of $\F_r$, and we may also assume $\alpha_n$ hyperfinite for all $n\in\mathbb{N}$ instead of periodic. The last assertion is due to the fact that an ascending union of hyperfinite equivalence relations is hyperfinite \cite[Theorem 6.11]{Kechris_Miller} . Thus, the subset of  the space of actions of a finitely generated group consisting of actions that generate a hyperfinite equivalence relation is  closed in the uniform topology.
\end{remark}
\begin{proof} (of Theorem \ref{thm:periodic_hf})
We start with the case $r = \infty$. Let
$$\SPAN = \{\alpha \in \Hom(F_{\infty},[E]) \ | \ \alpha(F_{\infty}) {\text{ spans the equivalence relation }}E \}.$$
By the Feldman-Moore theorem \cite{FM:equiv_rel}, \cite[Theorem 1.3]{Kechris_Miller} we can assume that the equivalence relation $E$ is generated by a different action $\beta \in \Hom(F_{\infty},[E])$. Fixing this action $\beta$ once and for all it is clear that
$\SPAN = \bigcap \SPAN(l,\Omega),$
where the intersection is over all finite subsets $\Omega \subset F_{\infty}$, all $l \in \N$ and $\SPAN(l,\Omega)$ is the set
\begin{small}
$$\left \{ \alpha \in \Hom(F_{\infty},[E]) \  \left | \ \mu \left \{ x \in X \ | \ \forall \omega \in \Omega, \exists \gamma \in F_{\infty} {\text{ such that }} \alpha(\gamma) x = \beta(\omega) x \right \} > 1 - \frac{1}{l} \right. \right \}.$$
\end{small}
We claim that every $\SPAN(l,\Omega)$ is open and hence $\SPAN$ is $G_{\delta}$. Indeed if $\alpha \in \SPAN(l,\Omega)$ then in fact it is so by virtue of only finitely many $\gamma \in \F_r$. Namely for such $\alpha$ we find a finite set $Q \subset \F_r$ such that
$$\mu \left \{ x \in X \ | \ \forall \omega \in \Omega, \exists \gamma \in Q{\text{ such that }} \alpha(\gamma) x = \beta(\omega) x \right \} > 1- \frac{1}{l}.$$
If $d_u(\alpha,\alpha')$ is small enough the same will be true for $\alpha'$. Which proves that $\SPAN(l,\Omega)$ is open. 

We turn to demonstrate density of $\SPAN$. Fix a basic open set in $\Hom(F_{\infty},[E])$, of the form $U = U(n;U_1,U_2,\ldots U_n) = \{ \alpha \in \Hom(F_{\infty},[E]) \ | \ \alpha(s_i) \in U_i \}$  . It is clear that $U \cap \SPAN \ne \emptyset$, as one can always choose $\{\alpha(s_i) \ | \ i > N\}$ in such a way that these span the whole equivalence relation. Thus we have shown that $\SPAN \subset \Hom(F_{\infty},[E])$ is a dense $G_{\delta}$, proving statement (\ref{itm:inf_gen}) of the Theorem.

From now on we focus on finite $r$. As the collection of finite graphs is clearly open in $\Sgr$, Lemma \ref{lem:Gd} implies that the set
$$\Fin(r) = \{\alpha \in \Hom(\F_r,[E]) \ | \ {\text{ almost all }} \alpha(\F_r) {\text{ orbits are finite}} \}$$
 is $G_{\delta}$. In order to prove claim (\ref{itm:hyperfinite}) of the theorem we have to show that for every $1<r \in \N$, $\Fin(r) \subset \Hom(\F_r,[E])$ is dense  if and only if $E$ is hyperfinite.

Assume first that $E$ is hyperfinite and let us realize it as an ascending union of finite equivalence relations $E = \cup F_j$. Setting
$$X(n,\gamma) = \left \{ x \in X \ | \ \alpha(\gamma)x F_n x \right\},$$
it is clear that $X(n,\gamma^{-1}) = \alpha(\gamma) \left(X(n,\gamma) \right)$ and that we can choose some $N \in \N$ such that $$\mu (X(N,s)) > 1 - \epsilon \qquad \forall s \in S.$$
For every $s \in S$ we define
$$\beta(s)(x) = \left \{ \begin{array}{ll}
\alpha(s)(x) & {\text{if }} x \in X(N,s) \\
\alpha(s^{-\xi(x,s)} x) & x \in X(N,s^{-1}) \setminus X(N,s) \\
x & {\text{otherwise}}
\end{array} \right. $$
Where $\xi(x,\gamma)  \in \N$ is defined to be the minimal integer such that  $\alpha(\gamma^{-\xi(x,\gamma)}) x \not \in X(N, \gamma^{-1})$. It is easy to check that this is well defined and gives us the desired $\beta$.

Conversely, assume that the periodic representations are dense in $\Hom(\F_r,[E])$. By Lemma \ref{lem:per_dense},  every element of $\Hom(\F_r,[E])$ generates a hyperfinite sub-equivalence relation of $E$. By the Feldman-Moore theorem \cite{FM:equiv_rel}, \cite[Theorem 1.3]{Kechris_Miller} we can assume that $E$ is generated by the action of a group $\Gamma$ generated by  $\{\phi_1,\phi_2,\ldots\}$. Let $E_n \subset E$ be the equivalence relation generated by the action of  $\langle \phi_1,\ldots, \phi_n \rangle$. By the above observation $E_2$ is hyperfinite (being generated by an element of $\Hom(\F_r,[E])$). Now we argue by induction that $E_n$ is hyperfinite for every $n$. Indeed if $E_n$ is hyperfinite then it is generated by the action of a single element $\langle \psi_n \rangle < [E]$ and $E_{n+1}$ is therefore generated by the action of the two generated group $\langle \psi_n,\phi_{n+1} \rangle$, so that the above observation still applies. As $E$ is the ascending union of the $E_n$ it follows that $E$ is also hyperfinite, as claimed.
\end{proof}

\section{Amenability}
This section is dedicated to the proof of Theorem \ref{thm:co-am}.
We wish to show that the set:
$$\Amm := \left \{ \alpha \in \HEr \ | \ \alpha(\\F_r)_x < \Gamma {\text{ is co-amenable for almost all }} x \in X \right \}$$
is residual in $\HEr$.
\begin{proof}
We will prove that this set is a dense $G_{\delta}$. The $G_{\delta}$ condition follows, using Lemma \ref{lem:Gd}, from the fact that the collection of Schreier graphs that contain a $1/l$-F\o lner set is open in $\Sgr$. 

Now for density. Given any $\alpha \in \HEr$ and $\epsilon > 0$ we seek $\beta \in \Amm$ with $d_u(\alpha,\beta) < \epsilon$. The idea of the proof uses a variation on a  construction by Vadim Kaimanovich \cite{Kaimanovich:amen_hf} further developed in \cite[Theorem 9.7]{Kechris_Miller}. Let us define $S:=\alpha(s_1)$, and note that by ergodicity and Remark \ref{rem:ap->erg} we may assume $S$ is ergodic.  Let $A \subset X$ be a subset such that $\mu(A)<\frac{\epsilon}{2r}$, and let $\{A_n\}_{n=1}^{\infty}$ be a Borel partition of $A$ (such that $\mu(A_n)>0$ for infinitely many $n\in\mathbb{N}$). The restriction of $E$ to $A_n$ is aperiodic and so by \cite[proposition 7.4]{Kechris_Miller} we may choose a finite sub-equivalence relation $F_n\subset E_{|A_n}$ such that every equivalence class in $F_n$ has cardinality $n$. For every $F_n$ choose a transversal $T_n$ and let $T:=\cup_{n=1}^{\infty}T_n$. Such a transversal exists because $F_n$ has finite classes, take for instance a Borel linear ordering on $A_n$ and let the transversal be the minimal element in each class. We now define $\beta$ as follows: Let $\beta(s_r)$ be an element of $[E]$ that generates $F_n$ on each $A_n$, and such that $d_u(\alpha(s_r),\beta(s_r))<\frac{\epsilon}{r}$, constructed as in \ref{lem:rep}. Now for $1 \le i \le r$ we define $\beta(s_i)$ to be the first return map of $\alpha(s_i)$ to the set $X \setminus (A \setminus T)$. Note that $d_u(\beta(s_i),\alpha(s_i)) \le \epsilon/r$ as these two transformations definitely coincide on $X \setminus (A \cup \alpha(s_i^{-1}) (A))$. Since $\alpha(s_1)$ is ergodic on $X$, $\beta(s_1)$ will be ergodic on $X \setminus (A \setminus T)$. In particular for every almost every $y \in X \setminus A$ the $\beta(s_1)$ orbit of $y$ intersects $T_n$ for infinitely many $n\in\mathbb{N}$ (these values of $n$ for which $\mu(T_n)>0$). Thus for almost every $y\in X$ the orbit of $y$ under $\beta(\F_r)$ contains an equivalence class of $F_n$ for infinitely many $n\in\mathbb{N}$. consequently,  we may choose the sequence of F\o lner sets to be those $F_n$ equivalence classes.
\end{proof}

\section{co-HT subgroups after Le-Ma\^itre}

Theorem \ref{thm:LM} follows directly from the theorem of Le-Ma\^itre on the existence of an $\alpha \in \Hom(\F_r,[E])$ with a dense image, coupled with Proposition \ref{prop:dense->ht}. The proof of this proposition relies on the following two lemmas that were provided, complete with their proofs, by the referee to whom we are grateful! They replace a similar argument that was both erroneous and less elegant.  

\begin{lemma} \label{lem:ref1} Let $X$ be a standard Borel space, let $T$ be a Borel bijection of $X$. Then there exists a Borel partition $(A_k)_{k\in\N}$ of the support of $T$ such that for all $k\in\N$, $T(A_k)$ is disjoint from $A_k$.
\end{lemma}
\begin{proof}
The only thing we use about the standard Borel space $X$ it that it has a countable separating family $\mathcal C$, i.e. for \textit{all} $x\neq y\in X$ there is $C\in\mathcal C$ such that $x\in C$ but $y\not\in C$. Replacing $\mathcal C$ by the union of the $T^n$ translates of $\mathcal C$, we can assume that $\mathcal C$ is $T$-invariant. Then the algebra $\mathcal B$ (no $\sigma$ here!) generated by $\mathcal C$ is countable $\Gamma$-invariant. 

Let $x\in \supp T$, because $\mathcal B$ separates points, we may find $C\in\mathcal B$ such that $x\in C$ but $T(x)\not\in C$. In other words, $x$ belongs to $B:=C\cap T^{-1}(X\setminus C)$, and because $\mathcal B$ is a $T$-invariant algebra  we have $B\in \mathcal B$. By definition $T(B)$ is disjoint from $B$. 

We thus have a covering $(B_k)_{k\in\N}$ of $\supp T$ by elements of $\mathcal B$ such that for all $k\in\N$, $T(B_k)$ is disjoint from $B_k$. We then let $A_k:=B_k\setminus(\bigcup_{i<k} B_k)\in\mathcal B$, and we have the desired partition\footnote{We adopt the convention that a partition may contain many times the empty set!}.
\end{proof}
\begin{lemma} \label{lem:ref2} Let $X$ be a standard Borel space, let $(T_i)_{i=1}^n$ be  Borel bijections of $X$. Then there exists a Borel partition $(A_k)_{k\in\N}$ of $\bigcap_{i=1}^n\supp(T_i)$ such that for all $k\in\N$ and $i\in\{1,...,n\}$, $T_i(A_k)$ is disjoint from $A_k$.
\end{lemma}
\begin{proof}
The proof is by induction; the base case $n=1$ is true by the previous lemma. Now let $(A_k)$ be a partition of $\bigcap_{i=1}^{n-1}\supp(T_i)$ obtained by induction hypothesis for $T_1$,...,$T_{n-1}$, and let $(A'_k)$ be a partition of the support of $T_n$ obtained by the previous lemma applied to $T:=T_n$. The partition $(A_k\cap A'_l)_{k,l\in\N}$ does the job.
\end{proof}

\begin{proof} (of proposition \ref{prop:dense->ht})
By the Feldman-Moore theorem \cite{FM:equiv_rel}, \cite[Theorem 1.3]{Kechris_Miller} we can assume that the equivalence relation $E$ is generated by an action of a countable group $G < [E]$. We will use the $G$ action to reference points in the various equivalence classes. Given $m$ distinct elements $g_0,g_2,\ldots, g_{m-1} \in G$ consider the set 
\begin{eqnarray*}
S_{g_0,g_1,\ldots,g_{m-1}}& = &  \bigcap_{0 \le i \ne j <m} \supp (g_i^{-1} g_j) \\
 & = &  \{x \in X \ | \ g_j x \ne g_i x, \ \ 0 \le i \ne j \le m-1\}.
 \end{eqnarray*} Using Lemma \ref{lem:ref2} we can find a countable partition Borel partition $S_{g_0,g_1,\ldots,g_{m-1}} = \sqcup_{k \in \N} (A_k)_{k\in\N}$  such that $\mu(g_i (A_k) \cap g_j (A_k)) = 0, \ \forall k \in \N {\text{ and }} 0 \le i \ne j \le m-1$. Given a permutation $\tau \in S_m$, considered as a permutation on $\{0,1,\ldots, m-1\}$, define an element $h_k$ of the full group $[E]$ by $h_{k|g_iA_k}=g_{\tau(i)}g_{i|g_iA_k}^{-1}$. (On the rest of $X$, $h_k$ may be defined arbitrarily). By definition, $\forall x\in A_k$, $h_kg_ix=g_{\tau(i)}x$. Let $\delta_k\in\Delta$ be such that $d_u(\delta_k,h_k)<\frac{\epsilon}{2^k}$, where $d_u$ denotes the uniform distance. Then we see that for every $x$ in a subset of $S_{g_0,\ldots, g_{m-1}}$ of measure at least $\mu(S_{g_0,\ldots, g_{m-1}})-\epsilon$ there exists an element $\delta\in\Delta$ such that 
 \begin{equation} \label{eqn:realize_tau}
 \delta g_ix=g_{\tau(i)}x.
 \end{equation}
As $\epsilon$ may be chosen to be arbitrarily small, this implies that for almost every $x\in S_{g_0,\ldots, g_{m-1}}$, there exists $\delta\in\Delta$ as above. Let $N_{g_0,\ldots, g_{m-1}; \tau} \subset S_{g_0,\ldots,g_{m-1}}$ be the subset of all points $x \in S_{g_0,\ldots,g_{m-1}}$ for which there does not exist an element $\delta \in \Delta$ satisfying Equation (\ref{eqn:realize_tau}). We have just proved that $N_{g_0,\ldots, g_{m-1}; \tau} $ is a nullset. Now conclude the proof by noting that the action of $\Delta$ is highly transitive on almost all orbits if and only if $N := \bigcap N_{g_0,\ldots, g_{m-1},\tau}$ is a nullset. Where the intersection here is taken over all possible values of $m \in \N$, all $m$-tuples of distinct elements $(g_0,g_1,\ldots,g_{m-1}) \subset G^m$ and all possible permutations $\tau \in S_m$. 
\end{proof}

\section{Higher transitivity}
Here we prove Theorem \ref{thm:coHT}. Fix our equivalence relation to be the ergodic hyperfinite equivalence relation $E_0$. We want to show that the set
$$\HT = \{\alpha \in \HEr \ | \ \alpha(\Gamma) {\text{ is highly transitive on almost all equivalence classes}} \},$$
is dense $G_{\delta}$
\begin{proof} 
Let $\sigma \in [E_0]$ be an automorphism generating the equivalence relation. We will use $\sigma$ to reference points in the various equivalence classes. Given $m \in \N$, a permutation $\tau \in S_m$ (acting on the set $\{0,1,\ldots, m-1\}$) and a point $x$ we say that a specific permutation representation $\alpha \in \HEr$ {\it{realizes the permutation $\tau$ at $x$}} if there exits some $\gamma \in \Gamma$ such that $\alpha(\gamma)$ acts as the given permutation on the first $m$ points in the equivalence class of $x$ as follows:
$$\alpha(\gamma): (x,\sigma x ,\sigma^2 x \ldots, \sigma^{m-1} x) \mapsto (\sigma^{\tau(0)} x, \sigma^{\tau(1)} x, \ldots, \sigma^{\tau(m-1)} x).$$
let $HT(\tau) = \{ \alpha \in \HEr \ | \ \alpha {\text{ realizes $\tau$ at almost every }} x \in X\}.$
By Baire's theorem it is enough to show that the set $\HT(\tau)$ is a dense $G_{\delta}$ subset of $\HEr$ for every $m \in \N$ and every $\tau \in S_m$. Since clearly $\HT = \bigcap_{m \in \N, \sigma \in S_m} \HT(\tau)$. The $G_{\delta}$ claim follows directly from Lemma \ref{lem:Gd}. 

To prove density of $\HT(\tau)$, let $\alpha \in \HEr$ and $\epsilon > 0$ be given, we seek $\beta \in \HT(\tau)$ with $d_u(\alpha,\beta) < \epsilon$. By \cite[Theorem 3.4]{Kechris:Global_aspects} the conjugacy class of any aperiodic element in $[E]$ is uniformly dense within the set of aperiodic elements of $[E]$, hence after replacing $\alpha$ by a very close representation and conjugating we may assume that $\alpha(s_1) = \sigma$. 

Now use Lemma \ref{lem:ref2} to find a set $O$ of measure $0<\mu(O) < \epsilon/2m$ such that $\{\sigma^i O \ | \ 0 \le i \le m-1 \}$ are pairwise disjoint.  By Lemma \ref{lem:rep}, we can find an element $\beta \in \HEr$ such that $\beta(s_2,\sigma^i x) = \sigma^{\tau(i)} x \ \forall x \in O$ and such that $d_u(\alpha,\beta) < \epsilon$. Now by ergodicity for almost every $x \in X$ one can find an $n \in \N$ such that $\sigma^n x \in O$. We choose $n = n(x)$ to be the minimal natural number satisfying this property. Now clearly
$$\beta(s_1^{-n(x)} s_2 s_1^{n(x)}, \sigma^i x) = \sigma^{\tau(i)} x, \ \ {\text{for almost every }} x \in X, {\text{ and }} \forall 0 \le i \le m-1.$$
Proving that $\beta \in \HT(\tau)$. This completes the proof of higher transitivity, using Baire's theorem.
\end{proof}

 \section{Core free}
This section is dedicated to the proof of Theorem \ref{thm:core}. Clearly a subgroup with a non-trivial core contains some non-trivial conjugacy class. It is enough to fix a conjugacy class
$C_{g} = \{\gamma g \gamma ^{-1} \ | \ \gamma \in \Gamma\}$ for some $g \ne e$ and show that the collection
\begin{eqnarray*}
\CF(g) & = & \{\alpha \in \HEr \ | \ \alpha(\F_r)_x \not \supset C(g) {\text{ for $\mu$-almost all }} x \in X\} \\
& = &  \{\alpha \in \HEr \ | \ \alpha(g) {\text{ acts non trivially on $\mu$-almost every orbit of }}\alpha\}
\end{eqnarray*} is residual, as
$$\CF = \{ \alpha \in \HEr \ | \ \alpha(\F_r)_x {\text{ is almost surely core free}} \} = \bigcap_{g \in \F_r \setminus \{e\}} \CF(g)$$
The fact that $\CF(g)$ is $G_{\delta}$ follows directly from Lemma \ref{lem:Gd}. 

Given $\alpha \in \HEr$ and $\epsilon > 0$ we look for $\beta \in \HEr$ such that $d_u(\alpha, \beta) < \epsilon$ and $\beta \in \CF(g)$. Namely we want $\beta(g)$ to act non-trivially on almost all $\beta$-orbits. The strategy to prove this is to define such $\beta$ with $\beta(s_1) = \alpha(s_1)$ and make sure that $\beta(g)$ acts non-trivially on almost all $\langle \alpha(s_1) \rangle = \langle \beta(s_1) \rangle$ orbits. Recall that in the lean aperiodic model $\alpha(s_1) = \beta(s_1)$ is, by assumption aperiodic. Let us set $\sigma = \alpha(s_1) = \beta(s_1)$. 

Since $\CF(s_1^l) = \HEr, \forall 0 \ne l \in \Z$ we can assume to begin with that $g \not \in C({s_1^l})$ for any such $l$. Also since we are interested only in the conjugacy class of $g$ we may assume that $g$ can be represented by a cyclicly reduced word $g = w_sw_{s-1} \ldots w_1$ where $w_i \in S \sqcup S^{-1}$. By our  assumption that $g \not \in C({s_1^l})$, not all letters are $s_1$ or $s_1^{-1}$. Using the Rokhlin Lemma,  we can find a set $O \in \Bc$ with $\mu(O) < \epsilon/2(s+1)$ such that the collection of sets $O, \sigma O, \sigma^2 O \ldots , \sigma^{s}O$ are disjoint. One can find a permutation $\tau \in S_{s+1}$ with the property that:
\begin{itemize}
\item whenever $w_i = s_1$ then $\tau(i) = \tau(i-1)+1$
\item whenever $w_i = s_1^{-1}$ then $\tau(i)=\tau(i-1)-1$.
\end{itemize}
We denote the disjoint union of these sets by $W = \cup_{i = 0}^s \sigma^i (O) = \cup_{i=0}^s \sigma^{\tau(i)}(O)$. There is an obvious element $\xi \in [E]$ that permutes the sets $\{\sigma^i(O) \ | \ 0 \le i \le s \}$ according the the new cyclic order induced on them by the permutation $\tau$. Namely
$$\xi(x) = \left \{ \begin{array}{ll} \sigma^{\tau(i+1)-\tau(i)}(x) & {\text{ if }} x \in \sigma^{\tau(i)}(O) \\
x & {\text{ if }} x \not \in W \end{array} \right.$$
Where $\tau(i+1)$ is understood as $\tau(0)$ when $i = s$. Now using Lemma \ref{lem:rep} we can find some $\beta \in \HEr$ such that $d_u(\alpha,\beta) < \epsilon$ but with the property that
\begin{equation} \label{eqn:instructions}
\beta(w_i)(x) = \xi(x) \quad \forall x \in \sigma^{\tau(i-1)}(O)
\end{equation}
There are only two things that one has to verify. First that there are no contradictions in the instructions prescribed in Equation \ref{eqn:instructions} which follows directly from the fact that the word was assumed to be reduced and hence we are prescribing the actions of these elements on distinct sets. The second is that whenever $w_i \in \{s_1,s_1^{-1}\}$ the prescribed action actually coincides with that of $\sigma, \sigma^{-1}$. But the permutation $\tau$ was chased exactly in order to accommodate that.

Now we can conclude the theorem. Things were arranged in such a way that $\beta(g) (\sigma^{\tau(0)}(O)) = \sigma^{\tau(s)}(O)$. And since, by ergodicity, $\sigma^{\tau(0)}(O)$ intersects almost every orbit, we are done.

\begin{remark}
It is possible to extend the statement of the theorem to arbitrary aperiodic equivalence relations. This can be done by finding a set $O$ in Rokhlin's lemma that has a positive measure with respect to almost every fiber measure in the ergodic decomposition. 
\end{remark}

\noindent {\bf Acknowledgments.}
Both authors were supported by ISF grant 441/11. Y.G. acknowledges support from U.S. National Science Foundation grants DMS 1107452, 1107263, 1107367 ``RNMS: Geometric structures And Representation varieties" (the GEAR Network). This work was written while Y.G. was on sabbatical at the University of Utah. I am  very grateful to hospitality of the math department there and to the NSF grants that enabled this visit. We are both very grateful to the referee whose excellent remarks improved the paper in numerous places. In particular for providing the two lemmas in section 5 and suggesting Corollary \ref{cor:conc}. 

\bibliographystyle{alpha}
\bibliography{../tex_utils/yair}

\newcommand{\etalchar}[1]{$^{#1}$}
\def\cprime{$'$} \def\cprime{$'$} \def\cprime{$'$} \def\cprime{$'$}
  \def\cprime{$'$} \def\cprime{$'$} \def\cprime{$'$} \def\cprime{$'$}
\begin{thebibliography}{ABB{\etalchar{+}}11}

\bibitem[ABB{\etalchar{+}}]{7_sam}
Miklos Abert, Nicolas Bergeron, Ian Biringer, Tsachik Gelander, Nikolay
  Nikolov, Jean Raimbault, and Iddo Samet.
\newblock On the growth of $l^2$-invariants for sequences of lattices in lie
  groups.
\newblock Preprint.

\bibitem[ABB{\etalchar{+}}11]{7_sam:ann}
Miklos Abert, Nicolas Bergeron, Ian Biringer, Tsachik Gelander, Nikolay
  Nikolov, Jean Raimbault, and Iddo Samet.
\newblock On the growth of {B}etti numbers of locally symmetric spaces.
\newblock {\em C. R. Math. Acad. Sci. Paris}, 349(15-16):831--835, 2011.

\bibitem[AGV13]{AGV:Kesten_Measureable}
Mikl\'{o}s Ab\'{e}rt, Yair Glasner, and B\'{a}lint Vir\'{a}g.
\newblock The measurable kesten theorem.
\newblock To appear in Annals of Probability. arXiv:1111.2080, 2013.

\bibitem[AGV14]{AGV:Kesten_IRS}
Mikl{\'o}s Ab{\'e}rt, Yair Glasner, and B{\'a}lint Vir{\'a}g.
\newblock Kesten's theorem for invariant random subgroups.
\newblock {\em Duke Math. J.}, 163(3):465--488, 2014.

\bibitem[BDL]{BDL:amenable_irs}
Uri Bader, Bruno Duchesne, and Jean L\'{e}cureux.
\newblock Amenable invariant random subgroups.
\newblock arXiv:1409.4745.

\bibitem[Bowa]{bowen:irs_free}
Lewis Bowen.
\newblock Invariant random subgroups of the free group.
\newblock arXiv:1204.5939.

\bibitem[Bowb]{bowen:furst_ent}
Lewis Bowen.
\newblock Random walks on random coset spaces with applications to furstenberg
  entropy.
\newblock To appear in Invent. Math.

\bibitem[Cha50]{Chabauty:top}
Claude Chabauty.
\newblock Limite d'ensembles et g\'eom\'etrie des nombres.
\newblock {\em Bull. Soc. Math. France}, 78:143--151, 1950.

\bibitem[FM77]{FM:equiv_rel}
J~Feldman and C.C. Moore.
\newblock Ergodic equivalence relations and von neumann algebras, i.
\newblock {\em Trans. Amer. Math. Soc.}, 234:289--324, 1977.

\bibitem[Gab00]{Gab:cout}
Damien Gaboriau.
\newblock Co\^ut des relations d'\'equivalence et des groupes.
\newblock {\em Invent. Math.}, 139(1):41--98, 2000.

\bibitem[Gla14]{G:linear_irs}
Yair Glasner.
\newblock Invariant random subgroups of linear groups.
\newblock arXiv:1407.2872, 2014.

\bibitem[GS]{GG:ht}
Yair Glasner and Garion Shelly.
\newblock Highly transitive actions of $\operatorname{Out}(f_n)$.
\newblock {\em to appear: Groups Geom Dyn. arXiv:1008.0563.}

\bibitem[HO15]{HO:ht}
Michael Hull and Dennis Osin.
\newblock Transitivity degrees of countable groups and acylindrical
  hyperbolicity.
\newblock arXiv:1501.04182, 2015.

\bibitem[Kai97]{Kaimanovich:amen_hf}
Vadim~A. Kaimanovich.
\newblock Amenability, hyperfiniteness, and isoperimetric inequalities.
\newblock {\em C. R. Acad. Sci. Paris S\'er. I Math.}, 325(9):999--1004, 1997.

\bibitem[Kec10]{Kechris:Global_aspects}
Alexander~S. Kechris.
\newblock {\em Global aspects of ergodic group actions}, volume 160 of {\em
  Mathematical Surveys and Monographs}.
\newblock American Mathematical Society, Providence, RI, 2010.

\bibitem[Kit12]{Kit:ht}
Daniel Kitroser.
\newblock Highly-transitive actions of surface groups.
\newblock {\em Proc. Amer. Math. Soc.}, 140(10):3365--3375, 2012.

\bibitem[KM04]{Kechris_Miller}
Alexander~S. Kechris and Benjamin~D. Miller.
\newblock {\em Topics in orbit equivalence}, volume 1852 of {\em Lecture Notes
  in Mathematics}.
\newblock Springer-Verlag, Berlin, 2004.

\bibitem[LM]{Le_Maitre}
F.~Le-Ma\^itre.
\newblock The number of topological generators for full groups of ergodic
  equivalence relations.

\bibitem[LM14]{LM:lean}
Fran\c{c}ois Le~Ma\^itre.
\newblock On full groups of non ergodic probability measure preserving
  equivalence relations.
\newblock arXiv:1405.4745, 2014.

\bibitem[McD77]{MD_free_HT}
T.~P. McDonough.
\newblock A permutation representation of a free group.
\newblock {\em Quart. J. Math. Oxford Ser. (2)}, 28(111):353--356, 1977.

\bibitem[MS13]{MS:ht}
Soyoung Moon and Yves Stalder.
\newblock Highly transitive actions of free products.
\newblock {\em Algebr. Geom. Topol.}, 13(1):589--607, 2013.

\bibitem[OW80]{OW}
D.~Ornstein and B.~Weiss.
\newblock Ergodic theory and amenable group actions, i: the rohlin lemma.
\newblock {\em Bull. Amer. Math. Soc.}, (2):161--164, 1980.

\bibitem[PT]{PT:stab_erg}
Jesse Peterson and Andreas Thom.
\newblock Character rigidity for special linear groups.
\newblock arXiv:1402.5028.

\bibitem[SZ94]{SZ}
Garrett Stuck and Robert~J. Zimmer.
\newblock Stabilizers for ergodic actions of higher rank semisimple groups.
\newblock {\em Ann. of Math. (2)}, 139(3):723--747, 1994.

\bibitem[TD]{TD:shift_minimal}
Robin Tucker-Drob.
\newblock Shift-minimal groups, fixed price 1, and the unique trace property.
\newblock Preprint.

\bibitem[Ver10]{Vershik:characters}
A.~M. Vershik.
\newblock Nonfree actions of countable groups and their characters.
\newblock {\em Zap. Nauchn. Sem. S.-Peterburg. Otdel. Mat. Inst. Steklov.
  (POMI)}, 378(Teoriya Predstavlenii, Dinamicheskie Sistemy, Kombinatornye
  Metody. XVIII):5--16, 228, 2010.

\end{thebibliography}

\noindent {\sc Amichai Eisenmann.} Department of Mathematics.
Ben-Gurion University of the Negev.
P.O.B. 653,
Be'er Sheva 84105,
Israel.
{\tt amichaie\@@math.bgu.ac.il}\bigskip

\noindent {\sc Yair Glasner.} Department of Mathematics.
Ben-Gurion University of the Negev.
P.O.B. 653,
Be'er Sheva 84105,
Israel.
{\tt yairgl\@@math.bgu.ac.il}\bigskip

\end{document}